\documentclass{amsart}
\usepackage{amsfonts}
\usepackage{graphicx}
\usepackage{amscd}
\usepackage{amsmath}
\usepackage{amssymb}
\usepackage{stmaryrd} 
\usepackage{color}

\makeatletter
\@namedef{subjclassname@2010}{%
  \textup{2010} Mathematics Subject Classification}
\makeatother

\theoremstyle{plain}

\newtheorem{corollary}{\bf Corollary}
\newtheorem{definition}{\bf Definition}
\newtheorem{lemma}{\bf Lemma}

\newtheorem{theorem}{\bf Theorem}

\theoremstyle{definition}

\numberwithin{equation}{section}

\title[Einstein type manifolds]{On Einstein-type manifold with \\ cyclic parallel Ricci tensor}
 
\author{M. Andrade} 
\author{H. Baltazar}
\author{A. da Silva$^\ast$}
\author{D. Tavares}

\address[M. Andrade]{Departamento de Matemática, 
 Universidade Federal de Sergipe, 49100-000, S\~ao Cristov\~ao-SE, Brazil.
}
\email{maria@mat.ufs.br}

\address{Current address: Department of Mathematics, Princeton University, Princeton, NJ, USA, 08544.}
\email{ma6208@princeton.edu}

\address[H. Baltazar]{Departamento de Matem\'{a}tica, Universidade Federal do Piau\'{\i}\\
64049-550 Te\-re\-si\-na, Piau\'{\i}, Brazil.}
\email{halyson@ufpi.edu.br}

\address[A. da Silva]{Faculdade de Matem\'{a}tica, Universidade Federal do Par\'a \\
66075-110, Bel\'em, Par\'a, Brazil.}
\email{adamsilva@ufpa.br}

\address[D. Tavares]{Programa de Doutorado em Matem\'{a}tica em Associação UFPA/UFAM, Universidade Federal do Par\'a \\
	66075-110, Bel\'em, Par\'a, Brazil.}
\email{diogomoan06@gmail.com}

\thanks{$^\ast$ Corresponding author.}

\begin{document}

\newcommand{\spacing}[1]{\renewcommand{\baselinestretch}{#1}\large\normalsize}
\spacing{1.2}

\begin{abstract}
In this article, we derive an integral formula involving the tensor $D_{ijk}$ for compact Einstein-type manifolds with constant scalar curvature. As an application, we classify three-dimensional compact Einstein-type manifolds satisfying the cyclic parallel Ricci tensor condition, obtaining rigidity results that extend and unify previous work in the literature.
\end{abstract}

\maketitle

\section{Introduction}\label{intro}
In this paper, \((M^n, g)\) denotes an \(n\)-dimensional, compact, connected, oriented smooth manifold with nonempty boundary \(\partial M\) and \(n \geq 3\). The study of Einstein metrics is a fruitful topic in differential geometry, and their generalizations have increasingly attracted the attention of the scientific community. In particular, Einstein manifolds play a central role due to their deep interplay with both mathematics and theoretical physics. In recent years, significant progress has been made toward the classification and geometric understanding of these spaces. For a comprehensive overview of the classical literature, we refer the reader to \cite{Besse}.

It is well known that every Einstein manifold has parallel Ricci tensor; however, the converse is not true in general, as illustrated by the standard cylinder. In the same direction, Gray \cite{Gray} introduced the notion of cyclic parallel Ricci tensor. More precisely, consider the \(3\)-tensor
\begin{equation}\label{TensorD}
D(X,Y,Z) = \nabla_X \operatorname{Ric}(Y,Z) + \nabla_Y \operatorname{Ric}(Z,X) + \nabla_Z \operatorname{Ric}(X,Y),  
\end{equation}
$\forall X,Y,Z \in \mathfrak{X}(M).$ A Riemannian manifold is said to have cyclic parallel Ricci tensor if \(D = 0\). Clearly, a parallel Ricci tensor implies \(D = 0\), but the converse does not hold in general (see \cite{Gray, Jelonek} for further details). A natural question arises: under what additional conditions does the converse hold, or does there exist a special geometric setting in which such an equivalence occurs?

In an attempt to address this question, we consider the so-called Einstein-type manifolds, which unify several well-known structures in the literature. These were introduced by Catino et al. \cite{Catino20}. Since then, many researchers have investigated their geometric properties, obtaining volume estimates and rigidity results; see, for instance, \cite{MBQ, AM24, Chen23, Chen24, Naza, L21, YH22, Yun}. Motivated by these works, we focus on a class of gradient Einstein-type metrics defined on manifolds with nonempty boundary, as described below.

\begin{definition}\label{def1}
An \textit{Einstein-type manifold} is a compact Riemannian manifold \((M^n, g)\), \(n \geq 3\), with smooth boundary \(\partial M\), for which there exist smooth functions \(f, h: M \to \mathbb{R}\) such that
\begin{equation}\label{fundEq}
\delta f \operatorname{Ric} + \nabla^2f = h g,
\end{equation}
where \(f > 0\) in \(\operatorname{int}(M)\), \(f^{-1}(0) = \partial M\), and \(h = \theta f + \gamma\) with constants \(\delta, \theta, \gamma \in \mathbb{R}\). Here, \(\operatorname{Ric}\) and \(\nabla^2f\) denote the Ricci tensor and the Hessian form on \(M^n\), respectively.
\end{definition}

We refer to equation \eqref{fundEq} as the fundamental equation of an Einstein-type manifold \((M^n, g, f, \delta, \gamma)\). This framework generalizes several important geometric equations. For instance, setting \(\delta = -1\) and \(u = -f\) yields \(f \operatorname{Ric} = \nabla^2 f + h g\), from which many relevant structures arise. In particular, if \(f\) is constant and \(\partial M = \emptyset\), then \((M^n, g)\) reduces to an Einstein manifold. Among the examples recovered by this formulation are the static vacuum Einstein equations. In the case of a null cosmological constant, they read
\[
f \operatorname{Ric} = \nabla^2 f, \quad \Delta f = 0,
\]
while for a non‑null cosmological constant they become
\[
f \operatorname{Ric} = \nabla^2 f + \frac{R f}{n-1} g, \quad \Delta f + \frac{R f}{n-1} = 0.
\]

For a static perfect fluid, the equations take the form
\[
f \operatorname{Ric} = \nabla^2 f + \frac{(\mu - \rho)f}{n-1} g, \quad \Delta f - \frac{(n-2)\mu + n\rho}{n-1} f = 0,
\]
where \(\mu\) and \(\rho\) denote the density and pressure, respectively, and satisfy the energy condition \(\mu \geq |\rho|\). Finally, the critical point equation of the total scalar curvature functional  is given by
\[
(1+f) \mathring{\operatorname{Ric}} = \nabla^2 f + \frac{R f}{n(n-1)} g, \quad \Delta f + \frac{R f}{n-1} = 0,
\]
with \(R\) denoting the scalar curvature (see \cite{MBQ, AM24} for more details).

In this context, it is worth recalling that critical metrics of the Einstein–Hilbert functional, also known as critical metrics of the total scalar curvature, are necessarily Einstein when considered over the space of all metrics. This result becomes even more refined when one restricts to the space of metrics with constant scalar curvature and unit volume. In this setting, it has been conjectured that every solution of the corresponding critical point equation is an Einstein metric (see \cite{Besse}).

A significant step in this direction was obtained in \cite{BS21}, where the authors proved the Besse conjecture for dimension three under the additional assumption of cyclic parallel Ricci tensor. More precisely, they established that, in this three‑dimensional context, critical metrics of the total scalar curvature functional with cyclic parallel Ricci tensor are indeed Einstein.

In another direction, Freitas and Gomes \cite{Naza} provided a classification of Einstein-type manifolds that are Einstein. Their result can be stated as follows.

\begin{theorem}\label{nazaTHM}\cite[Theorem 1, Theorem 2]{Naza}
Let \((M^n, g)\) be a compact gradient Einstein-type manifold with connected boundary \(\partial M\). If \((M^n, g)\) is an Einstein manifold, then it is isometric to a geodesic ball in a simply connected space form when \(\gamma \neq 0\), or to a hemisphere of a round sphere when \(\gamma = 0\).
\end{theorem}

Recently, Andrade, Baltazar, and Queiroz \cite{MBQ} investigated a special class of gradient Einstein-type manifolds under the parallel Ricci curvature condition, obtaining rigidity results. In \cite{BS21}, the authors studied three well‑known structures, namely, positive static triples, critical metrics of the volume functional, and critical metrics of the total scalar curvature functional, under the hypothesis of cyclic parallel Ricci tensor. These structures can all be viewed as particular instances of Einstein-type manifolds.

Before presenting our results, it is essential to introduce the following notation: 
for a given potential function \(f\) on \(M^n\), we denote $
i_{\nabla f} W = W(\cdot, \cdot, \cdot, \nabla f).$
A Riemannian manifold \((M^n, g)\) is said to have \textit{zero radial Weyl curvature} 
if the Weyl tensor satisfies \(i_{\nabla f} W = 0\). This class includes, in particular, 
locally conformally flat manifolds. Catino \cite{Catino12} employed this additional hypothesis to classify generalized 
quasi‑Einstein metrics with harmonic Weyl tensor and showed that this condition cannot 
be omitted.

Motivated by the aforementioned results, we now state our first contribution, which is an integral formula for Einstein-type manifolds, inspired by \cite[Theorem 8]{BS21}.
\begin{theorem}\label{ThmIF}
Let $(M^{n},g,f,\delta,\gamma)$ be a compact gradient Einstein-type manifold with constant scalar curvature and $\delta\in\mathbb{R}\setminus\{-1,\frac{1}{n-2}\}$. Then, we have that 
\begin{eqnarray*}
	\left(\frac{3\delta-\frac{2n-1}{n-2}}{\delta-\frac{1}{n-2}}\right)\frac{\delta}{9}\int_{M}f|D|^{2}\,dM
	&=&\frac{2\delta\left(\delta - \frac{n-1}{n-2}\right)}{\delta-\frac{1}{n-2}}\int_{M}f|\nabla \operatorname{Ric}|^{2}\,dM \\[4pt]
	&&-n\gamma\delta\int_{M}|\mathring{\operatorname{Ric}}|^{2}\,dM \\[4pt]
	&&-\frac{1+3\delta}{3}\int_{M}D_{ijk}R_{jk}\nabla_{i} f\,dM \\[4pt]
	&&+\frac{2}{3(\delta-\frac{1}{n-2})}\int_{M}\langle C,i_{\nabla f}W\rangle\,dM,
\end{eqnarray*}
\end{theorem}

As a consequence of this formula, we obtain the following rigidity result for a three-dimensional Einstein-type manifold that satisfies the cyclic parallel Ricci condition. 

\begin{theorem}\label{Thm3D}
	Let $(M^{3},g,f,\delta,\gamma)$ be a compact gradient Einstein-type manifold with $\delta\in\mathbb{R}\setminus\{-1,0,1\}$ and satisfying the cyclic parallel Ricci tensor condition. Then:
	\begin{itemize}
		\item[(i)] if $\gamma=0$ and $\delta\neq2$, then $(M^{3},g)$ has parallel Ricci tensor;
		\item[(ii)] if $\gamma>0$ and $1<\delta\leq2$, then $(M^{3},g)$ is an Einstein manifold;
		\item[(iii)] if $\gamma<0$ and $\delta<1$ or $\delta\geq2$, then $(M^{3},g)$ is an Einstein manifold.
	\end{itemize}
\end{theorem}

The next result is a direct consequence of Theorems~\ref{nazaTHM} and \ref{Thm3D}.

\begin{corollary}
	Let $(M^{3},g,f,\delta,\gamma)$ be a compact gradient Einstein-type manifold with connected boundary $\partial M$ and $\delta\in\mathbb{R}\setminus\{-1,0,1\}.$ Assume that either $\gamma>0$ and $1<\delta\leq2$, or $\gamma<0$ and $\delta<1$ or $\delta\geq2.$ If $(M^{3},g)$ satisfies the cyclic parallel Ricci tensor condition, then $(M^{3},g)$ is isometric to a geodesic ball in a simply connected space form.
\end{corollary}

Before announcing our last result, it is important to recall a result due independently to Kobayashi and Lafontaine. This result provides the classification for positive static triples that are locally conformally flat. For our purposes, we will state it in dimension 3.

\begin{theorem}[Kobayashi \cite{kobayashi}, Lafontaine \cite{lafontaine}]\label{KL}
	Let $(M^{3},g,f)$ be a $3$-dimensional positive static triple with scalar curvature $R=6$. Suppose that $(M^{3},g)$ has parallel Ricci tensor. Then $(M^{3},g,f)$ is covered by a static triple equivalent to one of the following two static triples:
	\begin{enumerate}
		\item The standard hemisphere with canonical metric $(\mathbb{S}^{3}_{+},g_{\mathbb{S}^{3}})$;
		\item The standard cylinder over $\mathbb{S}^{2}$ with the product metric
		$$\Big(M=\Big[0,\frac{\pi}{\sqrt{3}}\Big]\times\mathbb{S}^{2},\; g=dt^{2}+\frac{1}{3}g_{\mathbb{S}^{2}}\Big).$$
	\end{enumerate}
\end{theorem}

Therefore, considering the ideas developed in \cite[Theorem 2]{MBQ} jointly with the classification obtained in \cite{BS21}, we obtain the following result for an Einstein-type manifold with $\delta=-1$.

\begin{theorem}\label{ThmDelta-1}
	Let $(M^{n},g,f,\delta,\gamma)$ be a compact gradient Einstein-type manifold with $\delta=-1$, constant scalar curvature, and satisfying the cyclic parallel Ricci tensor condition. Then the following assertions hold:
	\begin{itemize}
		\item[(i)] if $\gamma=0$ and $n=3$, then $(M^{3},g)$ is covered by one of the static triples described in Theorem~\ref{KL};
		\item[(ii)] if $\gamma<0$, $n=3$, and $R>0$, then $(M^{3},g)$ is isometric to a geodesic ball in $\mathbb{S}^{3}$;
		\item[(iii)] if $\gamma<0$ and $R\leq0$, then $(M^{n},g)$ is isometric to a geodesic ball in $\mathbb{R}^{n}$ or $\mathbb{H}^{n}$.
	\end{itemize}
\end{theorem}

\section{Preliminaries}
\label{Preliminaries}

In this section, we review some basic facts and present some key results that will play a crucial role in the proof of our main results. We start by recalling that on a Riemannian manifold $(M^n,g)$, $n\geq 3$, the Riemann curvature tensor $Rm$ is given by the decomposition formula
\begin{equation}\label{weyl}
R_{ijkl}=W_{ijkl}+\frac{1}{n-2}(Ric\varowedge g)_{ijkl}-\frac{R}{2(n-1)(n-2)}(g\varowedge g)_{ijkl},
\end{equation}
where the symbol $\varowedge$ denotes the Kulkarni–Nomizu product defined for any two symmetric $(0,2)$-tensors $S$ and $T$ as follows:
$$(S\varowedge T)_{ijkl}=S_{ik}T_{jl}+S_{jl}T_{ik}-S_{il}T_{jk}-S_{jk}T_{il}.$$

Next, we have the Cotton tensor $C$, which is given by
\begin{equation}\label{cotton}
\displaystyle{C_{ijk}=\nabla_{i}R_{jk}-\nabla_{j}R_{ik}-\frac{1}{2(n-1)}\big(\nabla_{i}R\,g_{jk}-\nabla_{j}R\,g_{ik}\big).}
\end{equation}

It is easy to check that $C_{ijk}$ is skew-symmetric in the first two indices and trace-free in any two indices. 

We now obtain a result that is simple to verify but is essential for our results.

\begin{lemma}\label{riemannAUX}
	Let $(M^{n},g)$ be a Riemannian manifold. Then we have:
	\begin{itemize}
		\item[(i)] $R_{ij}R_{ik}R_{jk}=\operatorname{tr}(\mathring{Ric}^{3})+\dfrac{3}{n}R|\mathring{Ric}|^{2}+\dfrac{R^{3}}{n^{2}}$;
		\item[(ii)] $R_{ij}R_{ik}R_{jk}-R_{ijkl}R_{ik}R_{jl}=\dfrac{n}{n-2}\operatorname{tr}(\mathring{Ric}^{3})+\dfrac{R}{n-1}|\mathring{Ric}|^{2}-W_{ijkl}R_{ik}R_{jl}$.
	\end{itemize}
\end{lemma}

For what follows, let us recall the commutation formula for the Ricci tensor. In fact, for any Riemannian manifold $M^{n}$, we have
\begin{equation}\label{CRicci}
	\nabla_{i}\nabla_{j}R_{kl}-\nabla_{j}\nabla_{i}R_{kl}=R_{ijks}R_{sl}+R_{ijls}R_{ks}.
\end{equation}

The next result was obtained by the first author and Ribeiro Jr. in \cite{balt18}.

\begin{lemma}\label{keyDIV1}
	Let $(M^{n},g)$ be a Riemannian manifold with constant scalar curvature and let $f:M\rightarrow\mathbb{R}$ be a smooth function defined on $M$. Then we have
	\begin{eqnarray*}
		\frac{1}{2}\operatorname{div}(f\nabla|Ric|^{2})&=&-\frac{f}{2}|C_{ijk}|^{2}+f|\nabla Ric|^{2}+\frac{1}{2}\langle\nabla f, \nabla|Ric|^{2}\rangle\\
		&&+\nabla_{i}(fC_{ijk}R_{jk})+C_{ijk}R_{ik}\nabla_{j}f\\
		&&+\frac{n}{n-2}f\operatorname{tr}(\mathring{Ric}^{3})+\frac{1}{n-1}fR|\mathring{Ric}|^{2}-fW_{ijkl}R_{ik}R_{jl}.
	\end{eqnarray*}
\end{lemma}

Proceeding, we observe that the fundamental equation of an Einstein-type manifold can be rewritten in tensorial form as follows:
\begin{equation}\label{eq:tensorial}
\delta fR_{ij}+\nabla_{i}\nabla_{j}f=hg_{ij}.
\end{equation}
Tracing (\ref{eq:tensorial}) yields
\begin{equation}\label{eqtrace}
\delta fR+\Delta f=nh,
\end{equation}
where $\Delta$ is the Laplacian of $f\in C^{\infty}(M)$.

The first result that we would like to mention was proved by X. Chen in \cite[Lemma 2.3]{Chen23} considering a property restricted to the boundary. Here we assume that the manifold has constant scalar curvature and we obtain the same result, now valid for any point on the manifold, provided $\delta\neq-1$.

\begin{lemma}\label{Eric}
	Let $(M^{n},g,f,\delta,\gamma)$ be a compact gradient Einstein-type manifold with constant scalar curvature and $\delta\neq-1$. Then the gradient of $f$ is an eigenvector of the Ricci tensor. More precisely, we have
	$$Ric(\nabla f)=\frac{\delta R-(n-1)\theta}{1+\delta}\nabla f.$$ 
\end{lemma}

\begin{proof}
	First of all, differentiating (\ref{eq:tensorial}) and using the twice-contracted second Bianchi identity, we obtain
	$$\delta R_{ij}\nabla_{i}f+\frac{\delta f}{2}\nabla_{j}R+\nabla_{i}\nabla_{j}\nabla_{i}f=\theta\nabla_{j}f.$$
	This expression can be rewritten using the Ricci identity,
	\begin{equation}\label{Ricciiden}
		\nabla_{i}\nabla_{j}\nabla_{k}f-\nabla_{j}\nabla_{i}\nabla_{k}f=R_{ijkl}\nabla_{l}f,
	\end{equation}
	as
	\begin{equation}\label{AuxERic}
		(1+\delta)R_{ij}\nabla_{i}f=\frac{\delta f}{2}\nabla_{j}R+(\delta R-(n-1)\theta)\nabla_{j}f,
	\end{equation}
	where we have used (\ref{eqtrace}).
	
	Finally, using that the scalar curvature is constant, we obtain the desired result.	
\end{proof}

We observe that by differentiating both sides of (\ref{AuxERic}), we can deduce that 
\begin{eqnarray*}
	-(1+\delta)R_{ij}\nabla_{i}\nabla_{j}f =(-\delta R+(n-1)\theta)\Delta f+\frac{1-\delta}{2}\langle\nabla R,\nabla f\rangle-\frac{\delta}{2}\operatorname{div}(f\nabla R),
\end{eqnarray*}
which, combined with (\ref{eq:tensorial}), yields the following identity:
\begin{equation}\label{ricST}
	(1+\delta)\delta f|\mathring{\operatorname{Ric}}|^{2}=[R+(n-1)(n\theta-\delta R)]\frac{\Delta f}{n}+\frac{1-\delta}{2}\langle \nabla R,\nabla f\rangle-\frac{\delta}{2}\operatorname{div}(f\nabla R).
\end{equation}
This equality will be essential in our key lemmas.

The next result provides an expression for the Cotton tensor and can be found, for instance, in \cite{Chen24}.

\begin{lemma}\label{DeltaC}
Let $(M^{n},g,f,\delta,\gamma)$ be an Einstein-type manifold. Then we have:
\begin{eqnarray*}
	-\delta fC_{ijk}&=&R_{ijkl}\nabla_{l}f-(\nabla_{i}h\,g_{jk}-\nabla_{j}h\,g_{ik})+\delta(R_{jk}\nabla_{i} f-R_{ik}\nabla_{j}f)\\
	&&+\frac{\delta f}{2(n-1)}(\nabla_{i}R\,g_{jk}-\nabla_{j}R\,g_{ik}).
\end{eqnarray*}
In particular, if the scalar curvature is constant, then
\begin{eqnarray*}
	-\delta fC_{ijk}=R_{ijkl}\nabla_{l}f-(\nabla_{i}h\,g_{jk}-\nabla_{j}h\,g_{ik})+\delta(R_{jk}\nabla_{i} f-R_{ik}\nabla_{j}f).
\end{eqnarray*}
\end{lemma}

\begin{proof}
Computing $\nabla_i(\delta fR_{jk})$ in \eqref{eq:tensorial}, we obtain
\begin{equation}\label{deltaf}
\delta f\nabla_iR_{jk}=\nabla_i h\,g_{jk}-\nabla_i\nabla_j\nabla_kf-\delta\nabla_i fR_{jk}.
\end{equation}
Next, combining \eqref{cotton} and \eqref{deltaf} with the Ricci identity \eqref{Ricciiden}, we conclude the proof.
\end{proof}

Now, we recall the covariant $3$-tensor $T_{ijk}$ which appears naturally in the study of Einstein-type manifolds (see Section 2 in \cite{Chen24}). Its definition is as follows:\begin{eqnarray}\label{TensorT}
T_{ijk}&=&\frac{1}{n-2}(R_{jl}\nabla_{l}fg_{ik}-R_{il}\nabla_{l}fg_{jk})\nonumber\\
&&-\left[\frac{R}{(n-1)(n-2)}-\theta\right](\nabla_{j}fg_{ik}-\nabla_{i}fg_{jk})\nonumber\\
&&+\left(\frac{1}{n-2}-\delta\right)(R_{ik}\nabla_{j}f-R_{jk}\nabla_{i}f)\nonumber\\
&&+\frac{\delta f}{2(n-1)}(g_{jk}\nabla_{i}R-g_{ik}\nabla_{j}R).
\end{eqnarray}
It is important to highlight that $T_{ijk}$ was defined similarly in the of Miao–Tam critical metric (see \cite{BDR15}).

Note that from a straightforward computation, we observe that the tensor
$T$ has the same symmetries as the Cotton tensor $C$, i.e.,
$$T_{ijk} =-T_{jik} \text{ and }T_{ijk} + T_{jki} + T_{kij} = 0.$$
Moreover, in \cite{Chen23} Lemma 2.4 was proved that the tensor $T$ is related with Cotton and Weyl tensors by identity
\begin{equation}\label{CWT}
-\delta fC_{ijk}=W_{ijkl}\nabla_{l}f+T_{ijk}.
\end{equation}

\section{Key lemmas}
In this section we shall present a couple of lemmas that will be useful in the proof of our main results. The first part of these results will be presented in a similar way to the article \cite{MBQ}.

\begin{lemma}\label{keyDIV2}
Let $(M^{n},g,f,\delta,\gamma)$ be an Einstein-type manifold with constant scalar curvature. Then we have
\begin{eqnarray*}	
	\frac{\delta}{2}\operatorname{div}(f\nabla|\operatorname{Ric}|^{2})&=&2(1+\delta)C_{ijk}R_{ik}\nabla_{j}f-(1+\delta)\nabla_{i}(R_{ik}R_{jk}\nabla_{j}f)-\delta f|C_{ijk}|^{2}\\
	&&+\delta f|\nabla \operatorname{Ric}|^{2}+2\delta\nabla_{i}(fC_{ijk}R_{jk})+\frac{1+3\delta}{2}\langle\nabla f, \nabla|\operatorname{Ric}|^{2}\rangle\\
	&&-\frac{n-1}{n}R\Delta h+\mathring{R}_{ij}\nabla_{i}\nabla_{j}h+\delta\Delta f|\mathring{\operatorname{Ric}}|^{2}+\delta\frac{R^{2}}{n}\Delta f.
\end{eqnarray*}
\end{lemma}

\begin{proof}
Since the scalar curvature is constant, we may consider the field $X_{i}=R_{ik}R_{kj}\nabla_{j}f+R_{ijkl}\nabla_{l}fR_{jk}$ and use Lemma~\ref{DeltaC} to deduce
\begin{eqnarray}\label{FdivXaux}
	\operatorname{div}X &=& \nabla_{i}(R_{ik}R_{kj}\nabla_{j}f) \nonumber\\
	&&+\nabla_{i}\big[-\delta fC_{ijk}R_{jk}+R\nabla_{i} h-R_{ij}\nabla_{j}h -\delta|\operatorname{Ric}|^{2}\nabla_{i}f+\delta R_{ik}R_{kj}\nabla_{j}f\big] \nonumber\\
	&=&(1+\delta)\nabla_{i}(R_{ik}R_{kj}\nabla_{j}f)-\delta\nabla_{i}(fC_{ijk}R_{jk})+\frac{n-1}{n}R\Delta h -\mathring{R}_{ij}\nabla_{i}\nabla_{j}h \nonumber\\
	&&-\delta\langle \nabla f, \nabla|\operatorname{Ric}|^{2}\rangle-\delta|\mathring{\operatorname{Ric}}|^{2} \Delta f-\delta\frac{R^{2}}{n}\Delta f.
\end{eqnarray}

On the other hand, by direct computation using (\ref{eq:tensorial}), the fact that the scalar curvature is constant, and Lemma~\ref{DeltaC} again, we obtain
\begin{eqnarray}\label{divXaux}
	\operatorname{div}X &=& R_{ik}\nabla_{i}R_{jk}\nabla_{j}f+R_{ik}R_{kj}\nabla_{j}\nabla_{k} f \nonumber\\
	&&+\nabla_{i}R_{ijkl}\nabla_{l}fR_{jk}+R_{ijkl}\nabla_{i}\nabla_{l}fR_{jk}+R_{ijkl}\nabla_{l}f \nabla_{i}R_{jk} \nonumber\\
	&=&(1+\delta)R_{ik}\nabla_{i}R_{jk}\nabla_{j}f+C_{ijk}R_{ik}\nabla_{j}f-\frac{\delta f}{2}|C_{ijk}|^{2} \nonumber\\
	&&-\delta f(R_{ik}R_{ij}R_{jk}-R_{ijkl}R_{ik}R_{jl})-\frac{\delta}{2}\langle\nabla f,\nabla|\operatorname{Ric}|^{2}\rangle.
\end{eqnarray}

To proceed, notice that Lemma~\ref{keyDIV1} can be rewritten as 
\begin{eqnarray*}
	-\delta f(R_{ik}R_{ij}R_{jk}-R_{ijkl}R_{ik}R_{jl}) &=& -\frac{\delta f}{2}|C_{ijk}|^{2}+\delta f|\nabla \operatorname{Ric}|^{2}-\delta\nabla_{j}(fC_{ijk}R_{ik})\\
	&&-\frac{\delta}{2}\operatorname{div}(f\nabla|\operatorname{Ric}|^{2})+\delta R_{ik}C_{ijk}\nabla_{j}f\\
	&&+\frac{\delta}{2}\langle\nabla f,\nabla|\operatorname{Ric}|^{2}\rangle.
\end{eqnarray*}

Substituting this into (\ref{divXaux}) yields
\begin{eqnarray}\label{SdivXaux}
	\operatorname{div}X &=& 2(1+\delta)C_{ijk}R_{ik}\nabla_{j}f+\frac{1+\delta}{2}\langle\nabla f,\nabla|\operatorname{Ric}|^{2}\rangle-\delta f|C_{ijk}|^{2}+\delta f|\nabla \operatorname{Ric}|^{2} \nonumber\\
	&&+\delta \nabla_{i}(fC_{ijk}R_{jk})-\frac{\delta}{2}\operatorname{div}(f\nabla|\operatorname{Ric}|^{2}).
\end{eqnarray}

Finally, we simply compare expressions (\ref{FdivXaux}) and (\ref{SdivXaux}).
\end{proof}

The next lemma is a combination of Lemmas \ref{keyDIV1} and \ref{keyDIV2}. In fact, multiplying the result of Lemma \ref{keyDIV1} by $2\delta$ and subtracting it from Lemma \ref{keyDIV2}, we obtain

\begin{lemma}\label{keyDIV3}
	Let $(M^{n},g,f,\delta,\gamma)$ be an Einstein-type manifold with constant scalar curvature. Then we have
	\begin{eqnarray*}	
		\frac{\delta}{2}\operatorname{div}(f\nabla|\operatorname{Ric}|^{2})&=&\delta f|\nabla \operatorname{Ric}|^{2}-2C_{ijk}R_{ik}\nabla_{j}f -\frac{1+\delta}{2}\langle \nabla f, \nabla |\operatorname{Ric}|^{2}\rangle\\
		&&+2\delta f\left(\frac{n}{n-2}\operatorname{tr}(\mathring{\operatorname{Ric}}^{3})+\frac{R}{n-1}|\mathring{\operatorname{Ric}}|^{2}-W_{ijkl}R_{ik}R_{jl}\right)\\
		&&+\frac{n-1}{n}R\Delta h-\mathring{R}_{ij}\nabla_{i}\nabla_{j}h-\delta\Delta f |\mathring{\operatorname{Ric}}|^{2}-\delta\Delta f\frac{R^{2}}{n}\\
		&&+(1+\delta)\nabla_{i}(R_{ik}R_{jk}\nabla_{j}f).
	\end{eqnarray*}
\end{lemma}

To proceed, we will deduce an identity for the last two lines of the expression obtained in Lemma \ref{keyDIV3}. More precisely, we have the following result.

\begin{lemma}\label{keyID}
	Let $(M^{n},g,f,\delta,\gamma)$ be an Einstein-type manifold with constant scalar curvature and $\delta\neq-1$. Then
	$$\frac{n-1}{n}R\Delta h-\mathring{R}_{ij}\nabla_{i}\nabla_{j}h-\delta\Delta f |\mathring{\operatorname{Ric}}|^{2}-\delta\Delta f\frac{R^{2}}{n}+(1+\delta)\nabla_{i}(R_{ik}R_{jk}\nabla_{j}f)=-n\gamma\delta|\mathring{\operatorname{Ric}}|^{2}.$$
\end{lemma}

\begin{proof}
Firstly, using Lemma~\ref{Eric} and the fact that $h=\theta f+\gamma$, we may deduce 
\begin{eqnarray}\label{keyE1}
	&&\frac{n-1}{n}R\Delta h-\delta\Delta f \frac{R^{2}}{n}+(1+\delta)\nabla_{i}(R_{ik}R_{jk}\nabla_{j}f)\nonumber\\	
	&&=\left(\frac{n-1}{n}\theta R-\frac{\delta R^{2}}{n}\right)\Delta f+\frac{((n-1)\theta-\delta R)^{2}}{1+\delta}\Delta f\nonumber\\
	&&=((n-1)\theta-\delta R)\left(\frac{R}{n}+\frac{(n-1)\theta-\delta R}{1+\delta}\right)\Delta f\nonumber\\
	&&=((n-1)\theta-\delta R)\big[R+(n-1)(n\theta-\delta R)\big]\frac{\Delta f}{n(1+\delta)}.
\end{eqnarray}

Since we are considering constant scalar curvature, we may use (\ref{ricST}) to rewrite the last expression as 
\begin{eqnarray}\label{S1}
	\hspace{1.0cm}\frac{n-1}{n}R\Delta h-\delta\Delta f \frac{R^{2}}{n}+(1+\delta)\nabla_{i}(R_{ik}R_{jk}\nabla_{j}f)
	=((n-1)\theta-\delta R)\delta f|\mathring{\operatorname{Ric}}|^{2}.
\end{eqnarray}

At the same time, we can use (\ref{eq:tensorial}) and (\ref{eqtrace}) to infer 
\begin{eqnarray}\label{keyE2}
	-\mathring{R}_{ij}\nabla_{i}\nabla_{j}h-\delta\Delta f |\mathring{\operatorname{Ric}}|^{2}
	&=&-\theta\mathring{R}_{ij}(-\delta fR_{ij}+hg_{ij})-\delta\Delta f |\mathring{\operatorname{Ric}}|^{2}\nonumber\\
	&=&(\theta f-(-\delta f R+hn))\delta |\mathring{\operatorname{Ric}}|^{2}\nonumber\\
	&=&-((n-1)\theta-\delta R)\delta f|\mathring{\operatorname{Ric}}|^{2}-n\gamma\delta |\mathring{\operatorname{Ric}}|^{2}.
\end{eqnarray}

Adding expressions (\ref{S1}) and (\ref{keyE2}) yields the desired result.
\end{proof}

Finally, combining Lemmas \ref{keyDIV3} and \ref{keyID}, and using the fact that $f=0$ on $\partial M$, we obtain the last result of this section. This result is a key integral formula for deducing Theorem~\ref{ThmIF}.
\begin{lemma}\label{LastLemma}
	Let $(M^{n},g,f,\delta,\gamma)$ be an Einstein-type manifold with constant scalar curvature and $\delta\neq-1$. Then we have
	\begin{eqnarray*}	
		0&=&\delta \int_{M}f|\nabla Ric|^{2}dM-2\int_{M}C_{ijk}R_{ik}\nabla_{j}fdM -\frac{1+\delta}{2}\int_{M}\langle \nabla f, \nabla |Ric|^{2}\rangle dM\\
		&&+2\delta \int_{M}f\left(\frac{n}{n-2}tr(\mathring{Ric}^{3})+\frac{R}{n-1}|\mathring{Ric}|^{2}-W_{ijkl}R_{ik}R_{jl}\right)dM\\
		&&-n\gamma\delta\int_{M}|\mathring{Ric}|^{2}dM.
	\end{eqnarray*}
\end{lemma}

\section{Integral formula for Eintein-Type manifolds}

In this section we will demonstrate the Theorem~\ref{ThmIF} announced in the introduction. 
We start with two  identities involving the tensor $D_{ijk}$. In fact, by definition (\ref{TensorD}) we obtain
\begin{eqnarray}\label{Dijk}
	D_{ijk}R_{jk}\nabla_{i}f&=&\frac{1}{2}\langle\nabla f,\nabla|Ric|^{2}\rangle+2\nabla_{j}R_{ki}R_{jk}\nabla_{i}f\\
	&=&\frac{1}{2}\langle\nabla f,\nabla|Ric|^{2}\rangle+ 2(C_{jik}+\nabla_{i}R_{jk})R_{jk}\nabla_{i}f\nonumber\\
	&=&\frac{3}{2}\langle\nabla f,\nabla|Ric|^{2}\rangle+ 2C_{ijk}R_{ik}\nabla_{j}f\nonumber,
	\end{eqnarray}
where we have used the definition of Cotton tensor. For our purposes, we have rewritten it  as follows
\begin{eqnarray}\label{AuxD1}
	\frac{1}{2}\langle\nabla f,\nabla|Ric|^{2}\rangle=\frac{1}{3}D_{ijk}R_{jk}\nabla_{i}f- \frac{2}{3}C_{ijk}R_{ik}\nabla_{j}f.
\end{eqnarray}

Now, notice that the first line in \eqref{Dijk} and using \eqref{CRicci}, we  can be written as 
\begin{eqnarray}\label{AuxD2}
	D_{ijk}R_{jk}\nabla_{i}f&=&\frac{1}{2}\langle\nabla f,\nabla|Ric|^{2}\rangle+2\nabla_{i}(\nabla_{j}R_{ki}R_{jk}f)\nonumber\\
	&&-2\nabla_{i}\nabla_{j}R_{ki}R_{jk}f-2\nabla_{j}R_{ki}\nabla_{i}R_{jk}f\nonumber\\
	&=&\frac{1}{2}\langle\nabla f,\nabla|Ric|^{2}\rangle+2\nabla_{i}(\nabla_{j}R_{ki}R_{jk}f)\nonumber\\
	&&-2f(R_{ij}R_{ik}R_{jk}-R_{ijkl}R_{ik}R_{jl})\nonumber\\
	&&-\frac{1}{3}f|D_{ijk}|^{2}+f|\nabla Ric|^{2}.
\end{eqnarray}
Here, we have used the commutation formula mentioned in (\ref{CRicci}). Furthermore, the last step was used the expression that appears in the computation of the norm of tensor $D_{ijk},$ i.e., 
$$\frac{1}{3}|D_{ijk}|^{2}=D_{ijk}\nabla_{i}R_{jk}=|\nabla Ric|^{2}+2\nabla_{i}R_{jk}\nabla_{j}R_{ik}.$$

Continuing, we multiply expression (\ref{AuxD2}) by $\delta$ and integrate both sides to obtain
\begin{eqnarray*}\label{AuxD3}
	\delta\int_{M}D_{ijk}R_{jk}\nabla_{i}fdM
	&=&\frac{\delta}{2}\int_{M}\langle\nabla f,\nabla|Ric|^{2}\rangle dM-\frac{\delta}{3}\int_{M}f|D_{ijk}|^{2}dM+\delta\int_{M}f|\nabla Ric|^{2}dM\nonumber\\
	&&-2\delta\int_{M} f\left(\frac{n}{n-2}tr(\mathring{Ric}^{3})+\frac{R}{n-1}|\mathring{Ric}|^{2}-W_{ijkl}R_{ik}R_{jl}\right)dM
\end{eqnarray*}
where we have used Lemma~\ref{riemannAUX} item (ii) and $f=0$ on $\partial M$.
This expression combined with Lemma~\ref{LastLemma} becomes
\begin{eqnarray*}\label{AuxD3}
	\delta\int_{M}D_{ijk}R_{jk}\nabla_{i}fdM
	&=&-\frac{1}{2}\int_{M}\langle\nabla f,\nabla|Ric|^{2}\rangle dM-\frac{\delta}{3}\int_{M}f|D_{ijk}|^{2}dM+2\delta\int_{M}f|\nabla Ric|^{2}dM\nonumber\\
	&&-n\gamma\delta\int_{M}|\mathring{Ric}|^{2}dM-2\int_{M}C_{ijk}R_{ik}\nabla_{j}fdM.
\end{eqnarray*}
Futhermore, substituting (\ref{AuxD1}) in the above identity, we arrive at 
\begin{eqnarray}\label{AuxnormaD}
	\frac{3\delta+1}{3}\int_{M}D_{ijk}R_{jk}\nabla_{i}fdM
	&=&-\frac{\delta}{3}\int_{M}f|D_{ijk}|^{2}dM+2\delta\int_{M}f|\nabla Ric|^{2}dM\nonumber\\
	&&-n\gamma\delta\int_{M}|\mathring{Ric}|^{2}dM-\frac{4}{3}\int_{M}C_{ijk}R_{ik}\nabla_{j}fdM.
\end{eqnarray}

On the other hand, since we are considering $\delta\neq\frac{1}{n-2},$ we may use (\ref{TensorT}) and (\ref{CWT}) to deduce
$$C_{ijk}R_{ik}\nabla_{j}f=\frac{1}{2(\frac{1}{n-2}-\delta)}(-\delta f|C_{ijk}|^{2}+\langle C,i_{\nabla f}W\rangle).$$
Moreover, taking into account that $R$ constant, we can use the definition of Cotton tensor to get 
$$|C_{ijk}|^{2}=-\frac{1}{3}|D_{ijk}|^{2}+3|\nabla Ric|^{2}.$$
Thus, it is immediate to obtain
$$C_{ijk}R_{ik}\nabla_{j}f=\frac{1}{2(\frac{1}{n-2}-\delta)}\left(\frac{\delta}{3} f|D_{ijk}|^{2}-3\delta f|\nabla Ric|^2+\langle C,i_{\nabla f}W\rangle\right).$$
Substituting this data into (\ref{AuxnormaD}) we obtain the desired integral formula.

\subsection{Proof of Theorem~\ref{Thm3D}}
\begin{proof}
To begin with, taking into account our hypothesis, we apply Theorem~\ref{ThmIF} to deduce  
$$0=\frac{2(\delta-2)}{\delta-1}\int_{M}f|\nabla Ric|^{2}dM-3\gamma\int_{M}|\mathring{Ric}|^{2}dM.$$
Now, analyzing this identity considering  items (i), (ii) and (iii), we can observe that item (i) forces $M^{3}$ to have parallel Ricci tensor and itens (ii) and (iii) force $M^{3}$ to be Einstein, as we want to prove.
\end{proof}

\subsection{Conclusion of the Proof of Theorem~\ref{ThmDelta-1}}
\begin{proof}
Since the potential function $f$ is nontrivial, we may use  (\ref{AuxERic}) to conclude that 
\begin{equation}\label{AuxdDltaRTheta}
	\delta R-(n-1)\theta=0.
	\end{equation}
Next, considering the fundamental equation (\ref{eq:tensorial}) with $\delta=-1$, it is immediate that 
$$-\Delta f g+Hessf -fRic=(h-\Delta f)g=-(n-1)\gamma g,$$
where we have used (\ref{AuxdDltaRTheta}). This equation describes the well-known $V$-static metrics, and the classification will follow the results obtained in \cite{BS21}, more precisely, for item (i) we apply \cite[Theorem 4]{BS21}, for item (ii) we apply \cite[Theorem 6]{BS21} and finally, for item (iii) we apply \cite[Theorem 7]{BS21}.
\end{proof}

\section*{Data availability statement}
This manuscript has no associated data.

\section*{Conflict of interest statement}
On behalf of all authors, the corresponding author states that there is no conflict of interest.

\section*{Acknowledgments}
M. Andrade was partially supported by the Brazilian National Council for Scientific and Technological Development (CNPq, grants 408834/2023-4, 403869/2024-2, and 400078/2025-2) and FAPITEC/SE/Brazil (grant 019203.01303/2024-1). She also extends her gratitude to the Department of Mathematics at Princeton University, where this work was completed during her visit as a Visiting Fellow, for its warm hospitality. She is especially thankful to Ana Menezes for her continuous encouragement and support. H. Baltazar was partially supported by CNPq/Brazil [Grant: 420578/2025-0] and [Grant:302389/2022-9]. A. da Silva was partially supported by CNPq/Brazil [Grant: 420578/2025-0].

\end{document}